\documentclass[12pt]{amsart}
\usepackage[english]{babel}
\usepackage{amssymb,comment,stmaryrd,enumitem}

\usepackage{hyperref}
\newtheorem{prop}{Proposition}[section]
\newtheorem{lemma}[prop]{Lemma}
\newtheorem{lem}[prop]{Lemma}
\newtheorem{cor}[prop]{Corollary}
\newtheorem{thm}[prop]{Theorem}

\theoremstyle{definition}

\newtheorem*{rmk}{Remark}

\renewcommand{\d}{\mathbb{V} }
\newcommand{\N}{\mathbb{N}}
\newcommand{\R}{\mathbb{R}}
\newcommand{\E}{\mathbb{E}}

\newcommand\dinf[1]{\underline{\mathrm{d}}\big(#1\big)}
\newcommand\dsup[1]{\overline{\mathrm{d}}\big(#1\big)}
\newcommand\dens[1]{{\mathrm{d}}\big(#1\big)}
\begin{document}
\noindent

\title{On sum-product bases }

\author[F. Hennecart]{Fran\c cois Hennecart}
\address{F. Hennecart, Univ Lyon, UJM-Saint-\'Etienne, CNRS, ICJ UMR 5208, 42023
Saint-\'Etienne, France}
\email{francois.hennecart@univ-st-etienne.fr}

\author[G. Prakash]{Gyan Prakash} 
\address{G. Prakash, Harish-Chandra Research Institute, HBNI, Jhunsi, Prayagraj (Allahabad) -211 019, India}
\email{gyan@hri.res.in}

\author[E. Pramod]{E. Pramod}
\address{E. Pramod, Harish-Chandra Research Institute, HBNI, Jhunsi, Prayagraj (Allahabad) -211 019, India}
\email{epramod@hri.res.in}

\thanks{This paper  has been prepared and written within the framework of the IFCPAR/CEFIPRA project 5401-1}

\begin{abstract}
Besides various asymptotic results on the concept of sum-product bases in $\mathbb{N}_0$,
we consider by probabilistic arguments the existence of \emph{thin sets}  $A,A'$ of integers such that 
$AA+A=\mathbb{N}_0$ and $A'A'+A'A'=\mathbb{N}_0$.
 
\end{abstract}
\maketitle

\section{\bf Introduction}

Additive bases, and in less importance multiplicative bases, have been extensively studied for several centuries.
More recently, \emph{expanding polynomials} (of course with more than one variable) arise in this scope, whose point is to study the expansion of finite sets under polynomials. If $f\in\mathbb{Z}[x_1,x_2,\dots,x_d]$ and $A$ be contained in a given subset $R$ of a commutative ring, then let $f(A,A,\dots,A)$ (with $k$ arguments) denote the set of all terms $f(a_1,a_2,\dots,a_k)$ where the $a_i$'s come from $A$.
The polynomial $f$ is called an \emph{expander} if there exists $\delta>0$ such that $|f(A,\dots,A)|>|A|^{1+\delta}$ for any finite set $A$, where $|B|$ denotes the cardinality of a finite set $B$. If $R$ is finite, as for instance $\mathbb{F}_q$ or $\{1,\dots,N\}$, we need to restrict the above definition by assuming that $|R|^{\varepsilon}<|A|<|R|^{1-\varepsilon}$, for some $\varepsilon>0$.     
A more restrictive notion is the one of \textit{covering polynomial}: is there a non trivial \emph{minimal size} such that if $A$ attains it then $f(A,A,\dots,A)$ entirely covers $R$ ? 

We shall use the notation $AB$ to denote the set of elements $x$ such that $x=ab$ for some $a\in A$ and $b\in B.$  When $A=B$, we use the notation $A^2=AA$ and by extension $A^k=AA^{k-1}$, for $k>1$ with the convention $A^1=A.$
We shall focus on $R=\mathbb{N}_0$, the set of all nonnegative integers and the two special polynomials $x+yz$ and $xy+zt$ which  are known to be expanders in different contexts (cf. \cite{BO,HH0}).They also bring to light the important sum-product phenomenon. 
It can also be enlightened  by their ability to break the natural threshold for the size of a set $A$ satisfying $f(A,A,\dots,A)=R$ that can be deduced from the sum or the product taken separately.
More precisely and taking an instance, the set $A^2+A$ contains both $Aa_0+A$ and $A^2+a_0$ provided that $a_0\in A$. But we can expect to find sets $A$ such that $A^2+A=R$ which are \textit{much smaller}, with respect to their \emph{size}, than sets satisfying  $Aa_0+A=R$ or $A^2+a_0=R$.

We call $A$ to be a \emph{$f$-sum-product basis} for $R$ if $f(A,A,\dots,A)=R$. When $R$ is finite, the measure of the size $A$ of a set could be its cardinality. For infinite $R$, and mainly $\mathbb{N}_0$, we can use an appropriate notion of  counting function of a set $A$ or an appropriate notion of its density. 
\\\\
\textbf{Notation}. We let
$\mathbb{N}:=\mathbb{N}_0\smallsetminus\{0\}$ be the set of positive integers. \\ 
For $A\subset\mathbb{N}_0$ and $X>0$, let $A(X):=|A\cap[1,X]|$ and 
$$
\dinf{A}:=\liminf_{X\to\infty} X^{-1}A(X), \quad \dsup{A}:=\limsup_{X\to\infty}X^{-1}A(X),
$$
called respectively the \emph{lower density} and \emph{upper density} of $A$. We let $\dens{A}$ denote their common value if it is the case and call it the \emph{density} of $A$.\\
We shall use the symbols $\ll,\gg, \sim$ in the usual way.
The notation $g(x)\asymp f(x)$ means $f(x)\ll g(x)\ll f(x)$ for any $x$ large enough. All the implied constants in Vinogradov's symbol $\ll$  are generally absolute. If they depend upon $\varepsilon$, we write $\ll_{\varepsilon}$.

In this paper
we shall  study those subsets $A$ of natural numbers such that the set
\begin{equation*}
A^k +A^l
\end{equation*}
contains all sufficiently large natural numbers or at least has positive lower density, where $k,l$ are positive integers and $\max(k, l) \geq 2.$
Clearly if we want $A^2+A$ (resp. $A^2+A^2$) to cover all the positive integers, or at least to have a positive lower density, one needs $A(X)\gg X^{1/3}$ (resp.   $A(X)\gg X^{1/4}$).
Since there exist additive bases $B$ of order $2$ with counting function $B(X)\ll \sqrt{X}$, one may hope to find a set $A$ such that $A(X) = o(\sqrt{X})$ in both the particular discussed cases.
On the other hand, thin multiplicative bases of order $2$, that is sets $A$ satisfying $A^2=\mathbb{N}$, cannot be too small since they must contain all the primes, hence $A(X)\gg X/\log X$ (see the recent \cite{Pa} for recent progress on the subject). This 
suggests us that the gain below $\sqrt{X}$ cannot be more than a power of $\log X$.  In 
Section~\ref{gen_asymp}, we shall prove the following result.

\begin{thm}\label{limsuplb}
Let $k \geq l$ be positive integers with $k \geq 2$ and $A \subset \mathbb{N}$ such that the set $A^k+ A^l$ has a positive lower density. Then for infinitely many positive integers $X$, we have

$$ A(X) \gg \frac{\sqrt{X}}{\log^{\alpha(k,l)}X},$$ 

where $\alpha(k,l) =  \frac{k+l-2}{k+l}.$

\end{thm}

In Section~\ref{gen_asymp}, we shall also prove the following result.

\begin{thm}\label{ub}
There exists an $A\subset \N_0$ such that $A^k + A = \N_0$ and for all sufficiently large $X$, we have
$$
A(X) \ll \frac{\sqrt{X}}{\log^{\alpha(k)}X},
$$ where $\alpha(k) = \frac{k-2}{k+1}.$
\end{thm}

The probabilistic method remains an efficient method for proving the existence of thin bases  by controlling the asymptotic behaviour in a probabilistic way.
Nevertheless it could not provide optimally thin bases by a sufficiently general model. 

\medskip
In Section \ref{asymp_behav}, we study the possible deviation in the behaviour of the counting function $A(X)$ in the family of all  sets $A$ such that $A^2+A=\mathbb{N}_0$.  

\medskip\noindent
The existence of a set $A\subset\N_0$ such that $A^2+A=\mathbb{N}_0$ and $A(X)=o(\sqrt{X})$ is not yet solved.
We only mention that the dyadic set
$$
T=\{2\}\cup \left\{
\sum_{i=0}^k\varepsilon_i2^{2i},\ k\ge0,\ \varepsilon_i \in\{0,1\}
\right\}
$$
satisfies $T^2+T\supseteq 2\cdot T+T=\mathbb{N}$ and $\limsup_{X\to\infty}\frac{T(X)}{\sqrt{X}}=\sqrt{3}$. 
\\[0.5em]
In Section \ref{thin_almost}, we will show
\begin{thm}\label{locthin}
For any positive increasing function $\phi(X)$ going to infinity as $X \rightarrow \infty$, there exists a set $A\subset\N$ such that $\dens{A^2+A}=1$ and $\liminf_{X\to\infty}A(X)(X \phi(X))^{-1/3}<\infty$.
\end{thm}

\medskip\noindent
In Section \ref{basic-probab}, we give the necessary tools of probability theory. 

\medskip\noindent
In Section \ref{prob_S2+S2}, we construct a thin set $A$ such that $A^2+A^2=\mathbb{N}_0$ and whose counting function satisfies $A(X)=o(\sqrt{X})$. More precisely, we prove the following result.

\begin{thm}\label{thm13}
There exists $A \subset \N_0$ with $A^2 + A^2 = \N_0$ and $A(X) = O\left(\frac{\sqrt{X}}{\log^{1/4}X}\right).$
\end{thm}


\section{\bf General asymptotic bounds}
\label{gen_asymp}
In this section, we shall prove Theorems \ref{limsuplb} and \ref{ub}. For this we need the following result, which follows by partial summation.
\begin{lem} \label{produb} Let $\alpha, \beta < 1$ (not necessarily positive) be real numbers.
Let $A, B \subset \N$ such that $A(X) \ll \frac{\sqrt{X}}{\log^{\alpha}X}
$ and $B(X) \ll \frac{\sqrt{X}}{\log^{\beta}X}$ for all sufficiently large $X$. Then for all sufficiently large $X$, we have
\begin{equation}\label{ubAB}
(AB)(X) \leq \sum_{\substack{ a \in A, b\in B\\
 ab \leq X}} 1 \ll \sqrt{X}\log^{1- \alpha-\beta}X.
 \end{equation}
\end{lem}

We also have
\begin{lem}\label{prodlb}
Let $\alpha, \beta < 1$ be real numbers.
Let $A, B \subset \N$ such that $|A(X)| \gg \frac{\sqrt{X}}{\log^{\alpha}X}
$ and $B(X) \gg \frac{\sqrt{X}}{\log^{\beta}X}$ for all sufficiently large $X$. Then we have
\begin{equation*}
  \sum_{\substack{a \in A, b\in B\\
 ab \leq X}} 1 \gg \sqrt{X}\log^{ 1- \alpha-\beta}X.
 \end{equation*}
for all sufficiently large $X.$
\end{lem}

\begin{proof}[Proof of Lemma~\ref{produb}]
For any real number $X \geq 2,$
 we have $A(X) \leq c\frac{\sqrt{X}}{\log^{\alpha}(X)}$ and
$B(X) \leq c \frac{\sqrt{X}}{\log^{\beta}(X)}$ for some $c> 0.$

\vspace{0.2cm}
\noindent
Therefore
\begin{eqnarray}\label{ub1}
\sum_{\substack{a \in A, b\in B\\
 ab \leq X}} 1 &\leq & A(2)B(X) + B(2)A(X) + \sum_{\substack{2\leq a \leq \frac{X}{2}\\ 
 a \in A}} B\Big(\frac{X}{a}\Big)\\
 & \leq & c \sqrt{X}\sum_{2\leq a \leq \frac{X}{2}}  \frac{I_A(a)}{\sqrt{a} \log^{\beta}(\frac{X}{a})} + O\left(\frac{\sqrt{X}}{\log^{\min(\alpha, \beta)}X } \right).\nonumber
 \end{eqnarray}
 
 \noindent
  By partial summation we obtain
\begin{eqnarray}\label{ub2} 
 \sum_{2\leq a \leq \frac{X}{2}}  \frac{I_A(a)}{\sqrt{a} \log^{\beta}(\frac{X}{a})}
 &= &
\frac{1}{\log^{\alpha}(X) \log^{\beta}2} +
\int_2^{\frac{X}{2}} A(t)\frac{\log(\frac{X}{t}) -\beta }{2t^{3/2}\log^{\beta +1}(\frac{X}{t})}dt\\
& \ll & 
\int_2^{\frac{X}{2}} \frac{dt}{t\log^{\beta}(\frac{X}{t})\log^{\alpha}(t)} +O(\log^{-\alpha}(X)).\nonumber
\end{eqnarray}
Since $1-\alpha > 0$, $1-\beta > 0,$ we obtain
\begin{equation}
\int_2^{\frac{X}{2}} \frac{dt}{t\log^{\beta}(\frac{X}{t})\log^{\alpha}(t)} \ll  B(1-\beta, 1-\alpha) \log^{1- (\alpha + \beta)}X, \label{ub3}
\end{equation}
where $B$ is the beta function.  We obtain~\eqref{ubAB} from~\eqref{ub1}, \eqref{ub2} and \eqref{ub3}.
\end{proof}
\noindent
A similar argument gives Lemma~\ref{prodlb}. We provide the details below. 

\begin{proof}[Proof of Lemma~\ref{prodlb}.]
There exists $c > 0$ and a real number $X_0$  such that $A(X) \geq c\frac{\sqrt{X}}{\log^{\alpha}(X)}$ and
$B(X) \geq c \frac{\sqrt{X}}{\log^{\beta}(X)}$ for all $X \geq X_0.$ Therefore for all sufficiently large $X$, we have

\begin{equation}\label{prodlbeq}
\sum_{\substack{a \in A, b\in B\\
 ab \leq X}} 1 \geq \sum_{\substack{X_0\leq n \leq \sqrt{X}\\
 n\in A}} B\Big(\frac{X}{n}\Big) \geq c \sqrt{X} \int_{X_0}^{\sqrt{X}} A(t)\frac{\log(\frac{X}{t}) -\beta }{2t^{3/2}\log^{\beta +1}(\frac{X}{t})}dt + O(\sqrt{X}\log^{-\alpha}(X)).
\end{equation}
When $t \in [X_0, \sqrt{X}]$, we have $\log(\frac{X}{t}) -\beta \geq \frac{\log(\frac{X}{t}) }{2}$ for all sufficiently large $X.$ Using this 
 we obtain that
 \begin{eqnarray*}
 \int_{X_0}^{\sqrt{X}} A(t)\frac{\log(\frac{X}{t}) -\beta }{2t^{3/2}\log^{\beta +1}(\frac{X}{t})}dt &\gg& \int_{X_0}^{\sqrt{X}}\frac{1}{t\log^{\beta }(\frac{X}{t})\log^{\alpha}t}dt \\
 &=& \log^{1-(\alpha+\beta)}X\int_{\frac{\log X_0}{\log X}}^{\frac12}\frac{du}{(1-u)^{\beta}u^{\alpha}}\\
 &\geq & \log^{1-(\alpha+\beta)}X \int_{\frac14}^{\frac12} \frac{du}{(1-u)^{\beta}u^{\alpha}},
 \end{eqnarray*}
 provided $X \gg X_0^4.$
 Using this and~\eqref{prodlbeq} the claim follows. 
\end{proof}

\begin{cor} \label{A^nub} Let $\alpha < 1$ be a real number and $n \geq 2$ be an integer. Let
$A \subset \N$ such that $A(X) \ll \frac{\sqrt{X}}{\log^{\alpha}X}
$ for all sufficiently large $X$. Then for all sufficiently large $X$, we have
$$ A^k(X)  \ll \sqrt{X}\log^{k-1- k\alpha}X.$$ 
\end{cor}
\begin{proof}
Using induction, this is an immediate corollary of Lemma~\ref{produb}. 
\end{proof}

Theorem~\ref{limsuplb} now follows from Corollary~\ref{A^nub} and the inequality
$AB(X) \leq A(X)B(X),$ which is easy to verify..

\vspace{0.5cm}
For proving Theorem~\ref{ub}, we need the following result.
\begin{lem}\label{p1}
Let $\alpha > 0$ be a real number and $P$ be the set of primes. Then there exists a set $P_1 \subset P$ such that for any sufficiently large integer $X$, we have 
$$
\frac{\sqrt{X}}{\log^{\alpha}X} \ll P_1(X) \ll\frac{\sqrt{X}}{\log^{\alpha}X}.
$$ 
In fact, we also have $P_1 \cap (0.5X, X] \gg \frac{\sqrt{X}}{\log^{\alpha}X}.$
\end{lem}
\begin{proof}
For any sufficiently large natural number $n$, we have $|P \cap (n,2n]| \geq \frac{n}{2\log n}.$ We choose any $P_1\subset P$ which satisfies that  for all sufficiently large natural numbers $l$,  $|P_1 \cap (2^l, 2^{l+1}]|  = \left[\frac{2^{l/2}}{2\log^{\alpha} 2^l}\right].$ Then $P_1$ is a set as required.
\end{proof}

\begin{cor}\label{P^nlb} Let $\alpha < 1$ be a real number. Let $P_1$ be a subset of primes with $P_1(X) \geq c\frac{\sqrt{X}}{\log^{\alpha}X}$ for any sufficiently large real number $X$, where $c> 0$ is a constant. Then for any $k \geq 2$ we have
$$P_1^k(X) \geq c_1 \sqrt{X}\log^{k-1 -k\alpha}X$$
for any sufficiently large $X$ with $c_1 > 0$ being a constant depending only on $c$ and $k.$

\end{cor}
\begin{proof}
The claim is trivial for $k=1$. Suppose it is true for $k = l-1$ with $l\geq 2.$ Let $A = P_1^{l-1}$ and $B = P_1$. For any natural number if $r(n)$ denotes the number of solutions $(a,b)$ of $n = ab$ with $a \in A$ and $b \in B$, then $r(n) \leq l$. Hence we have 

$$\sum_{\substack{a \in A, b\in B\\
ab \leq X}} 1 = \sum_{\substack{ n \in AB\\
n \leq X}} r(n) \leq l P_1^l(X).$$
Using the above inequality and applying Lemma~\ref{prodlb}, the claim follows.
\end{proof}

We need the following result due to Lorentz.

\begin{thm}\cite[page 13, Theorem 6]{HR} \label{Lorentz} Let $A \subset \N$  with at least $2$ elements. Then there exists an additive complement $B\subset \N$ of $A$, namely such that $\mathbb{N}\smallsetminus(A+B)$ is finite,  with
$$
B(X) \ll \sum_{\substack{ n =1\\
A(n) > 0}}^X \frac{\log A(n)}{A(n)}.
$$ 
\end{thm}

We now give a proof of Theorem~\ref{ub}.

\begin{proof}[Proof of Theorem~\ref{ub}] Let $\alpha = \alpha(k)$ be as in Theorem~\ref{ub}.
For this $\alpha$, let $P_1$ be as in Lemma~\ref{p1}. Then using Corollary~\ref{P^nlb}, we have $P_1^k(X) \gg
\sqrt{X}\log^{k-1 -k\alpha}X.$ Then using Theorem~\ref{Lorentz}, there exists an additive complement $B$ of $P_1^k$ with $B(X) \ll \frac{\sqrt{X}}{ \log^{k-2-k\alpha}X}.$ We obtain the result by taking $A = P_1 \cup B$ and noticing that our choice of $\alpha$ satisfies $\alpha = k-2 -k\alpha.$
\end{proof}

\section{\bf Asymptotic behaviour for sets $A$ such that $A^2+A=\mathbb{N}_0$}
\label{asymp_behav}

In this section we give an account on the deviation for the counting function (beforehand normalized) of sets $A$ such that 
$A^2+A=\mathbb{N}_0$.

Let ${ A }\subset \mathbb{N}_0$ and define
$$
\alpha_A=\inf\left\{t\ge 0\,:\, \liminf_{X\to\infty}\frac{A(X)}{X^t}<\infty\right\},\quad
\beta_A=\inf\left\{t\ge 0\,:\, \limsup_{X\to\infty}\frac{A(X)}{X^t}<\infty\right\}.
$$ 
\begin{prop}
Let $ A $ such that $ A ^2+ A =\mathbb{N}_0$. Then
\begin{enumerate}[label={\rm (\alph*)}]
\item $\alpha_{ A }\ge 1/3$, 
\item $\beta_{ A }\ge 1/2$,
\item $\alpha_{ A }+\beta_{ A }\ge1$.
\end{enumerate}
\end{prop}

\begin{proof}
(a) We must have for any positive real number
$$
X\le \sum_{\substack{a,b,c\in A\\ ab+c\le X}}1\le A^2(X)A(X)\le A(X)^3
$$
hence $\alpha_{ A }\ge1/3$.\\[1em]
(b) Let $\beta>\beta_{ A }$. Then $A(X)\le X^{\beta}$ for any $X$ large enough.
It follows that
\begin{align*}
A^2(X)& \leq \sum_{\substack{a,b\in { A }\\ ab\le x}}1\le \sum_{\substack{a\in { A }\\a\le X}}A\left(\frac Xa\right)\ll X^{\beta}\sum_{\substack{a\in { A }\\a\le X}}a^{-\beta}
= X^{\beta}\left(\frac{A(X)}{X^{\beta}}+\beta\int_1^X\frac{A(t)}{t^{\beta+1}}dt\right)\\ 
&\ll X^{\beta}\left(1+\beta\int_1^X t^{-1}dt\right)\ll X^{\beta}\log X.
\end{align*}
Thus $X\le A(X)A^2(X)\le X^{2\beta}\log X$, yielding $\beta \geq 1/2$, whence $\beta_{ A }\ge1/2$.\\[1em]
(c) Let $\alpha>\alpha_{ A }$ and $\beta>\beta_{ A }$ and $X$ large enough such that
$A(X)\ll X^{\alpha}$. We also have $A(Y)\ll Y^{\beta}$ for any $Y$. Thus 
$A^2(X)\ll X^{\beta}\log X$ and $$X^{\alpha}\gg A(X)\gg \frac{X}{A^2(X)}\gg \frac{X^{1-\beta}}{\log X}.$$
It follows that $\alpha+\beta\ge1$ for any $\alpha>\alpha_{ A }$ and any $\beta>\beta_{ A }$. Hence $\alpha_{ A }+\beta_{ A }\ge1$.
\end{proof}

We now prove the reverse statement:
\begin{prop}\label{conv-alphabeta}
For any pair of real numbers $(\alpha,\beta)$ satisfying 
$$0\le 1-\beta\le \alpha\le \beta < 1 \text{ and }\alpha\ge1/3,$$ there exists
${ A }\subset\mathbb{N}_0$ such that ${ A }^2+{ A }=\mathbb{N}_0$ 
and
\begin{eqnarray}
A(X) & \ll & X^{\beta} \log^{2/3} X, \nonumber\\
A(X) & \gg & X^{\alpha} \log^{1/3} X,  \quad\text{and}\nonumber\\
A(X) & \ll & X^{\alpha} \log^{1/3} X \quad\text{for infinitely many natural numbers } X. \nonumber
\end{eqnarray}
In particular, we have $(\alpha,\beta)=(\alpha_{ A },\beta_{ A })$.
\end{prop}

\begin{proof}
Let $x_1 \geq 64$ be a sufficiently large natural number so that for any real number $x \geq x_1,$ we have
\begin{equation}\label{eqq1}
\pi(x, 2x) \gg \frac{x}{\log x},
\end{equation}
where $\pi(x, 2x)$ denotes the number of primes in the interval $(x, 2x].$ Let $\{x_1, x_2, \ldots \}$ be the sequence of natural numbers defined by $x_{i+1} = x_{i}^7$, $i \geq 1.$ For any $i \geq 1,$ let $y_i = x_{i+1}^{\frac{\alpha}{\beta}}$ be a real number. Then $y_i \geq x_{i+1}^{1/3} \geq 4 x_i^2.$ Let $P_1$ be a subset of primes with following properties:
\begin{eqnarray}
P_1 \cap [x_i, x_{i+1}] &\subset & (x_i, y_i],\  \forall i,\nonumber\\
\left|P_1 \cap [2^j x_i, 2^{j+1} x_i]\right|& \asymp & {(2^{j+1}x_i)^{\beta} \log^{1/3}x_i},\; 1\leq j \leq \log_2\frac{y_i}{x_i} -1,\  \forall i.\nonumber
\end{eqnarray}
By \eqref{eqq1}, since $\beta < 1$, there exist such $P_1.$ Let $A_1 = P_1 \cup \{0,1\}.$ It is easy to verify that we have
\begin{eqnarray}
A_1(t) &\gg & t^{\beta}\log^{1/3} t, \; 2x_i \leq t \leq y_i, \; \forall i,\nonumber\\
A_1(t) & \ll & t^{\beta}\log^{1/3} t,\ \forall t,\nonumber\\
A_1(t) & \gg & t^{\alpha} \log^{1/3} t,\ \forall t,\nonumber\\
A_1(x_i) & \ll & x_i^{\alpha} \log^{1/3} x_i, \ \forall i.\nonumber
\end{eqnarray} 
Here all the implied constants in the above inequalities are absolute. Let $t \geq x_2$ be a real number. Then  $t \in (x_i, x_{i+1}]$ for some $i \geq 2.$ Using the above inequalities, we have
\begin{eqnarray}
A_1^2(t) & \gg & t^{2\beta} \log^{2/3} t, \text{ if } 4x_i^2 \leq t \leq y_i^2,\nonumber\\
A_1^2(t) & \gg & t^{2\alpha} \log^{2/3} t, \text{ if } t \in (x_i, 2x_i] \cup (y_i^2, x_{i+1}],\nonumber\\
A_1^2(t) &\gg & t^{\beta} \log^{1/3} t, \text{ if } 2x_i \leq t \leq 4x_i^2.\nonumber
\end{eqnarray}
In particular, we have  $A_1^2(t) \gg t^{\min(2\alpha, \beta)}\log^{1/3}t$. Using Theorem~\ref{Lorentz}, there exists $B \subset \mathbb{N}_0$ such that $A_1^2 + B = \mathbb{N}_0$ and
$B(t) \ll t^{\alpha}\log^{2/3}t$. Moreover we have for any $i \geq 1$,
\begin{equation*}
B(x_{i+1}) \ll x_i^2 + \sum_{4x_i^2 \leq t \leq x_{i+1}} \frac{\log A_1^2(t)}{A_1^2(t)} \ll x_{i+1}^{1-2\alpha}\log^{1/3}x_{i+1} \ll x_{i+1}^{\alpha} \log^{1/3}x_{i+1}.
\end{equation*}
Let $A = B \cup A_1.$ Then we have $A^2 + A = \mathbb{N}_0$ and $A(x_i) \ll x_i^{\alpha} \log^{1/3} x_i$. We have 
\begin{equation*}
\max(A_1(t), B(t)) \leq A(t) \leq A_1(t)+ B(t).
\end{equation*}
Using this, the result follows. In case $\alpha < \beta$, in fact we also have $A(t) \ll t^{\beta} \log^{1/3}t$ for every~$t$.
\end{proof}

In \cite[Theorem 1.8]{HH} the authors proved that for any $n$ there exists a finite set $S\subset\mathbb{N}_0$ such that
$|S|\ll (n\log n)^{1/3}$ and $\{0,1,\dots,n\}\subset S^2+S$. We can extend the idea of \cite{HH} to show the following: 

\begin{cor}\label{liminf}
There is an infinite  set $A_0\subset\mathbb{N}_0$ such that 
$$
A_0^2+A_0=\mathbb{N}_0 \quad\text{ and }\quad \liminf_{X\to\infty}\frac{A_0(X)}{(X\log X)^{1/3}}<\infty.
$$
\end{cor}
\begin{proof}
Using Proposition~\ref{conv-alphabeta}, with $\alpha = 1/3$ and $\beta = 2/3$, the result follows.
Note that we also have  $$\limsup_{X\to\infty}\frac{A_0(X)}{(X\log X)^{2/3}}<\infty.$$
\end{proof}
\section{\bf Some basic results in probability}\label{basic-probab}
Let $Y = \{0,1\}^{\N}.$ Any set $A \subset \N$ is in one-one correspondence with its indicator function which is an element of $Y.$ One can show an existence of a set $A \subset \N$ satisfying certain properties by assigning a suitable probability measure on $Y$ (that is collection of all subsets of $\N$) such that the probability of collection of those subsets of $\N$ which satisfy the required properties is strictly positive. In Sections~\ref{thin_almost}, 
and \ref{prob_S2+S2}, we shall use this method to show an existence of a set with the properties we are interested in.

\vspace{.5cm}
Now $\{0,1\}$ is a discrete topological space and $Y$ is a product topological space. Let $\mathcal{B} \subset \mathcal{P}(Y)$ be the Borel  $\sigma$-algebra on $Y.$ Given any sequence of real numbers $\{x_a \}_{a\in \N}$ with
$0\leq x_a \leq 1,$ let $p_a:\mathcal{P}(\{0,1\}) \to [0,1]$ be a sequence of probability measure such that $p_a(\{1\}) =x_a.$ Then there exists a unique probability measure $\mathbb{P} : \mathcal{B} \to [0,1]$ such that $\mathbb{P} = \prod_{a\in \N} p_a.$ One says that we are selecting a \emph{random subset} $A$ of $\N$ by selecting every element $a\in N$ with probability $x_a$ and the elements are selected independently.
We shall write $\mathbb{E}^{\mathbb{P}}(Z)$ (or simply $\mathbb{E}(Z)$) and $\mathbb{V}^{\mathbb{P}}(Z)$ (or simply $\mathbb{V}(Z)$)  respectively  for the expectation and the variance of a random variable $Z$ on this probability space.

\vspace{.5cm}
For any $a\in \N$, let $\xi_a : X \to \{0,1\}$ be the projection to the $a$-th coordinate and we define
\begin{equation*}
A(n) = \sum_{a \leq n} \xi_a, \; \text{ and } \lambda_n = \sum_{a \leq n} x_a.
\end{equation*}
Then the following result is an easy corollary  of \cite[Corollary 1.10]{TV}.

\begin{lemma}\label{chernoff} For any $0 < \varepsilon < 1/2,$ we have
$$
\mathbb{P}\left(\{A \subset \N: (1-\varepsilon)\lambda_n \leq A(n) \leq (1+\varepsilon)\lambda_n\}\right) \geq 1- 2 \exp^{-\frac{\varepsilon^2 \lambda_n}{4}}.
$$
\end{lemma}

If there exists a finite set $C\subset \N$ such that for $a \notin C$, we have $y_a =0$, then  $\mathbb{P}$ induces a probability measure on $\{0,1\}^C$ and
for any $0 < \varepsilon < 1/2,$ we have
\begin{equation}\label{finite-chernoff}
\mathbb{P}\left(\{A \subset C: (1-\varepsilon)\lambda \leq |A| \leq (1+\varepsilon)\lambda\}\right) \geq 1- 2\exp\left(-\frac{\varepsilon^2 \lambda}{4}\right),
\end{equation}
 where $\lambda = \sum_{a\in C} y_a.$
 
\begin{lemma}[Borel-Cantelli Lemma]
 Let $E_n \in \mathcal{B}$ with $\sum_{n}\mathbb{P}(E_n) < \infty.$ Then we have
\label{Borel-Cantelli}
\begin{equation*}
\mathbb{P}\left(\{A\subset \N: A \notin E_n \text{ for all sufficiently large } n\}\right) = 1
\end{equation*}

\end{lemma}

\begin{cor}\label{card}
Suppose $\lambda_n \ge \kappa \log^2 n$ then 
\begin{equation*}
\mathbb{P}\left(\{A\subset \N: \lim_{n\to\infty}{\lambda_n^{-1}}{A(n)}=1\}\right) = 1
\end{equation*}
\end{cor}
\begin{proof}
We choose $\varepsilon=\frac{8}{\sqrt{\kappa \log n}}$ in Lemma \ref{chernoff}. This implies that the probability that
$\big|\frac{A(n)}{\lambda_n}-1\big|\gg \frac1{\log n}$ is $O(\frac1{n^2})$. We conclude by the Borel-Cantelli Lemma. 
\end{proof}

\bigskip
Let $R(n)$ be a sequence of random variables on $Y.$ In our applications, we shall need to show that for almost every set, $R(n) \neq 0$ for all sufficiently large $n.$ The following result is an immediate corollary of Lemma~\ref{Borel-Cantelli}.
\begin{lemma}\label{R(n)gen}
Let $R(n)$ be a sequence of random variables on $Y.$ If $\mathbb{P}(\{R(n) =0\}) \leq \frac{1}{n^{1+\eta}}$ for some fixed $\eta > 0$, then we have 
\begin{equation*}
\mathbb{P}(\{A \subset \N: R(n) \neq 0 \text{ for all sufficiently large } n\}) =1.
\end{equation*}
\end{lemma}
 We assume that $R(n)$ depends only upon the first $n$ coordinates. Then $R(n)$ may be viewed as a random variable on $Y_n = \{0,1\}^n.$  Moreover $\mathbb{P}$ induces a probability measure $\mathbb{P}_n= \prod_{i=1}^n p_i$ on $Y_n$ and 
\begin{equation*}
\mathbb{P}(\{R(n)= 0\}) = \mathbb{P}_n(\{R(n) = 0\}).
\end{equation*}
In order to obtain an upper bound for the probability of those sets such that $R(n) = 0$, we shall use Janson's inequality. Before stating it, we need some assumptions on $R(n)$ and some notations. 

\vspace{.5cm}

For any $n$, we shall assume that there exist a finite index set $I$ and for every $i\in I$ a Boolean random variable $Z_i$ on $Y_n$ such that
\begin{equation*}
R(n) = \sum_{i\in I} Z_i. 
\end{equation*}
Let $\Gamma$ be a simple undirected graph with vertex set as the elements of $I$ without loop and if $(i,j)\notin \Gamma$, then we assume that $Z_i$ and $Z_j$ are independent. Let
\begin{equation*}
\mu_n = \E(R(n))\; \text{ and } \Delta_n = \sum_{(i,j) \in \Gamma} \E(Z_iZ_j).
\end{equation*}

\begin{lemma}[Janson's inequality]
\label{Janson}
We have
\begin{equation*}
\mathbb{P}_n(R(n) = 0) \leq \exp{(-\mu_n +\Delta_n)}.
\end{equation*}
\end{lemma}

\bigskip
A function $f: Y_n \to \R$ is said to be monotone increasing function if $f(x_1,\ldots, x_n) \leq f(y_1,\ldots, y_n)$, whenever $x_i \leq y_i \; \forall i.$ In our applications $I_{R(n)\neq 0}$ will be a monotone increasing function. The following result shall be useful in obtaining an upper bound for $\mathbb{P}_n(\{R(n) =0\}).$

\begin{lemma}\label{transfert}
Let $f: Y_n \to \R$ be a monotone increasing function. Let $\mathbb{P}_n= \prod_{i=1}^n p_i$ and $\mathbb{P}'_n= \prod_{i=1}^n p'_i$ be two probability measure on $Y_n$ with $p_i(\{1\}) \geq p'_i(\{1\})\; \forall i.$ Then
\begin{equation*}
\mathbb{E}^{\mathbb{P}_n}f: = \sum_{y\in Y_n} f(x)\mathbb{P}_n(\{y\}) \geq \sum_{y\in Y_n}f(y)\mathbb{P}'_n(\{y\}): = \mathbb{E}^{\mathbb{P}'_n}f.
\end{equation*}

\end{lemma}
\begin{proof}
We first show the result when there exists an $i_0$ with $1 \leq i_0 \leq n$ such that
$p_{i} = p'_i$ for every $i \neq i_0.$ We may assume, without any loss of generality that $i_0 = 1.$ Then we have
\begin{multline}\label{expec-identity}
\E^{\mathbb{P}_n}f = \sum_{y \in \{0,1\}^{n-1}}f(y,0)\bigg(\prod_{i=2}^{n}p_i \bigg)(\{y\}) \\+ \sum_{y\in \{0,1\}^{n-1}}(f(1,y)-f(0,y))p_1(\{1\})\bigg(\prod_{i=2}^{n}p_i \bigg)(\{y\}).
\end{multline}
Since $f$ is monotone increasing, for any $y \in \{0,1\}^{n-1}$, we have $f(y,1)-f(y,0)\geq 0.$  Hence we have
\begin{equation}\label{mono}
\left(f(y,1)-f(y,0)\right)p_1(\{1\})\geq \left(f(y,1)-f(y,0)\right)p'_1(\{1\}).
\end{equation}
Using \eqref{expec-identity} and \eqref{mono}, we obtain the result when $p_{i} = p'_{i}$ for any $i \neq 1.$ Using the induction hypothesis, we may assume that the result holds when the number of $i$ such that $p_i \neq p'_i$ is at most $k \geq 1.$
If $k=n$, then we have nothing to prove. If $k < n$, we need to show that the result holds when  the number of $i$ such that $p_i \neq p'_i$ is equal to $k+1.$ Without any loss of generality, we may assume that $p_i = p'_i$ for every $i \geq k+2.$ Let $\mathbb{P}''_n = \prod{i=1}^np''_i$ be the measure on $Y_n$ with $p''_i =p'_i$ for $i \leq k$ and $p''_i = p_i$ for $i \geq k+1.$ Using the induction hypothesis, we have
\begin{equation*}
\E^{\mathbb{P}_n}f \geq \E^{\mathbb{P}''_n}f \geq \E^{\mathbb{P}'_n}f.
\end{equation*}
Hence the result follows.
\end{proof}

\section{\bf Locally extremely thin almost sum-product basis}
\label{thin_almost}

In Corollary~\ref{liminf}, it was shown that there exists $A \subset \mathbb{N}_0$ with $A^2 + A = \mathbb{N}_0$ and $A(X) \ll (X\log X)^{1/3}$ for infinitely many integers $X.$
To obtain a thinner set in the sense that $A(X)\ll X^{1/3}$ for infinitely many integers $X$  is out of reach. Nevertheless it happens that by relaxing 
the covering condition $A^2+A=\mathbb{N}_0$ into $\dinf{A^2+A}>1-\varepsilon$, we can obtain  
 such a set $A$ satisfying $A(X)\ll_{\varepsilon} X^{1/3}$ for infinitely many integers $X$ (cf. Theorem \ref{thm53}). 
 
 \medskip
We shall use the ideas from an \emph{additive complement lemma} for finite sets of integers due to Ruzsa (see \cite[Lemma 2.1]{Ru}). We state and prove the needed version.

\begin{lem}\label{neolem}
Let $0<\varepsilon<\frac12$ be sufficiently small, $n\in\N$ and $A\subset\N$ such that $n> 10^5 \varepsilon^{-9/2}$ and 
\begin{equation}\label{hypo}
\forall x\in\llbracket n^{1/3},\varepsilon n\rrbracket,\ \forall  m\in\llbracket 2x,2\varepsilon^{-1}x\rrbracket,\quad
\big|A\cap\llbracket m-2x,m-x\rrbracket\big|> {\varepsilon}x^{2/3}\log\Big(\frac nx\Big).
\end{equation}
Then there exists $B\subset \llbracket n^{1/3}, 2\varepsilon n\rrbracket$ such that $|B|\ll \varepsilon^{-2/3}n^{1/3}$ and 
$$
\forall t\in\llbracket 2n^{1/3},2n\rrbracket,\quad \big|\llbracket 2n^{1/3},t\rrbracket\smallsetminus(A+B)\big|\le \varepsilon t.
$$
\end{lem}

\begin{proof}
Let $C = \llbracket n^{1/3}, 2\varepsilon n\rrbracket.$ We define a probability measure $\mathbb{P} = \prod_{a\in C}p_a$ on $Y = \{0,1\}^C$ by choosing

$$y_a : = p_a(\{1\}) = \frac{10{\varepsilon}^{-1}}{a^{2/3}}.$$
Our assumption implies that $y_a < 1$ and hence there exists such a probability measure.
 Then 
$$
\lambda:=\sum_{a \in C} y_a \asymp \varepsilon^{-2/3}n ^{1/3}
$$

Using~\eqref{finite-chernoff}, we have
$$
\mathbb{P} \left(B \subset C: |B| \ge 2\lambda\right)\leq 2 \exp
\left(-\frac{\lambda}{4}\right)
$$
which can be made smaller than $1/4$ by choosing $\epsilon$ small enough. Hence 
\begin{equation}\label{with}
\text{with probability at least $3/4$, $|B|\ll  \varepsilon^{-2/3} n^{1/3}$.}
\end{equation}
For any $B \subset C,$ we denote $B_j=B\cap\llbracket \varepsilon^{j+1}n,2\varepsilon^{j+1} n\rrbracket$
for any $0\le j \le   J_{\varepsilon}:=\lceil \frac{\log n^{2/3}}{\log \varepsilon^{-1}}\rceil-1$. Let 
$m\in\llbracket 2\varepsilon^{j+1}n,2\varepsilon^{j} n\rrbracket$. Then since 
\begin{equation}\label{nn22}
\forall a\in \llbracket
m-2\varepsilon^{j+1}n,m-\varepsilon^{j+1}n
\rrbracket,\ 
y_{m-a}\ge\frac{10\varepsilon^{-1}}{(2\varepsilon^{j+1}n)^{2/3}}>p:=\frac{5\varepsilon^{-1}}{(\varepsilon^{j+1}n)^{2/3}},
\end{equation}
we get
\begin{align*}
\mathbb{P}\left(m-a\not\in B_j,\ \forall a\in A\cap\llbracket
m-2\varepsilon^{j+1}n,m-\varepsilon^{j+1}n
\rrbracket\right) &\le (1-p)^{|A\cap\llbracket
m-2\varepsilon^{j+1}n,m-\varepsilon^{j+1}n
\rrbracket|}\\
&\le \exp\left(-p|A\cap\llbracket
m-2\varepsilon^{j+1}n,m-\varepsilon^{j+1}n \rrbracket|\right).
\end{align*}
By \eqref{hypo} and \eqref{nn22} this gives
\begin{align*}
\mathbb{P}\left(m-a\not\in B_j,\ \forall a\in A\cap\llbracket
m-2\varepsilon^{j+1}n,m-\varepsilon^{j+1}n
\rrbracket\right)
& \le \exp\left( -5(j+1){\log\varepsilon^{-1}}\right)\\
&\le \frac{\varepsilon^2}{8(j+1)^2}.
\end{align*}
We infer
$$
\mathbb{E}\left(
\llbracket 2\varepsilon^{j+1}n,2\varepsilon^{j} n\rrbracket\smallsetminus (A+B_j)
\right)\le \frac{\varepsilon^{j+2}n}{4(j+1)^2},
$$
hence by Markov's inequality
$$
\mathbb{P}\left(
\big|\llbracket 2\varepsilon^{j+1}n,2\varepsilon^{j} n\rrbracket\smallsetminus (A+B_j)\big|
>\varepsilon^{j+2}n
\right)\le \frac{1}{4(j+1)^2},
$$
and finally
$$
\mathbb{P}\left(\exists j, \ 0\le j \le J_{\varepsilon}\,:\,
\big|\llbracket 2\varepsilon^{j+1}n,2\varepsilon^{j} n\rrbracket\smallsetminus (A+B_j)\big|
>\varepsilon^{j+2}n
\right) < \frac{1}{2}.
$$
With \eqref{with}, we deduce  that there exists a set $B$ such that $|B|\ll  \varepsilon^{-2/3} n^{1/3}$ and 
 $$
 \forall j, \ 0\le j \le J_{\varepsilon}\,:\,
\big|\llbracket 2\varepsilon^{j+1}n,2\varepsilon^{j} n\rrbracket\smallsetminus (A+B_j)\big|
\le \varepsilon^{j+2}n.
 $$
 Now let $t>2n^{1/3}$. Then there is a $j$ with $0\le j \le J_{\varepsilon}$ such that 
 $2\varepsilon^{j+1}n< t \le  2\varepsilon^j n$. Hence 
 $$
 \left|\llbracket 2n^{1/3},t\rrbracket\smallsetminus (A+B)\right|\le 
 \sum_{i=j}^{J_{\varepsilon}}\varepsilon^{i+2}n\le 2 \varepsilon^{j+2}n\le \varepsilon t.
 $$
This ends the proof of the lemma. 
\end{proof}

We deduce the main result of the section.

\begin{thm}\label{thm53}
For any $\varepsilon>0$, there exists an infinite sequence $A_0$ of integers such that 
$$
\dinf{A_0^2+A_0}\ge 1-{\varepsilon}\quad \text{and}\quad \liminf_{n\to\infty} X^{-1/3}A_0(X)\ll 
\varepsilon^{-5/6}
\sqrt{\log \varepsilon^{-1}}\ll \varepsilon^{-1}
$$
where the implied constant is absolute.
\end{thm}

\begin{proof}
We can plainly assume $0<\varepsilon<\frac12$.
Let $(N_k)_{k\ge1}$ be a sequence of integers where $N_1 \varepsilon^{9/2}$ is big enough and $N_{k+1}=N_k^3$. 
This implies that 
\begin{equation}\label{bigen}
N_{2k+1}^{1/9}\varepsilon^{9/2}\quad \text{is big enough for any $k\ge1$.}
\end{equation}
In order to apply Lemma \ref{neolem}, we define our sufficiently big set $A$ according to 
hypothesis \eqref{hypo}. 

\bigskip
Let $k\ge1$ and $N=N_{2k+1}$.
Firstly we define a set of prime numbers $P_{2k+1}\subset [\sqrt{\varepsilon N^{1/3}},\sqrt{2N}]$.
Let $j$ be an integer such that 
$$
0\le j\le \left\lceil\frac{\log N^{2/3}}{\log\varepsilon^{-1}}\right\rceil-1.
$$ 
We split the interval $[\varepsilon^{j+1}N ,2\varepsilon^{j}N ]$
into $O(\varepsilon^{-2})$ intervals of size $\frac{\varepsilon^{j+2}N }2$.
If for some $2\varepsilon^{-1}\le r\le 4\varepsilon^{-2}$  
$$
I_{j,r}:=\left[\frac{r\varepsilon^{j+2}N }2, \frac{(r+1)\varepsilon^{j+2}N }2\right]
$$
is such an interval, then  the 
interval
$$
\left[\sqrt{\frac{r\varepsilon^{j+2}N }2}, \sqrt{\frac{(r+1)\varepsilon^{j+2}N }2}\right]
$$
has length $\gg {\frac{\sqrt{r\varepsilon^{j+2}N }}{{r}}}$, hence by the Prime Number Theorem it
contains at least
$ \frac{\varepsilon^{\frac j2+1}\sqrt{N }}{\sqrt{r}\log N }
~\gg~\frac{\varepsilon^{\frac j2+2}\sqrt{N }}{\log N }$ many primes. 

\medskip
We observe that the above remains true when $\varepsilon$ tends to $0$ when $N$ increases to infinity, as 
for instance $\varepsilon > (\log\log N)^{-1}$. We shall use this fact in the proof of Theorem \ref{locthin}.  
\medskip

We have $\varepsilon^{-j}\le N^{2/3}$ hence by condition \eqref{bigen} 
$$
\varepsilon^{\frac j2+2}\frac{\sqrt{N }}{\log N }\ge 2 \sqrt{{\varepsilon} (\varepsilon^{j+2}N )^{2/3}\log \varepsilon^{-(j+1)}}.
$$
We thus may assign 
$$
\left\lceil  2 \sqrt{{\varepsilon} (\varepsilon^{j+2}N )^{2/3}\log \varepsilon^{-(j+1)}}\right\rceil
$$
prime numbers into $P_{2k+1}$. Arguing similarly for each interval $I_{j,r}$, we obtain the required  sequence of primes  $P_{2k+1}$.

\medskip
Our aim is now to show that hypothesis \eqref{hypo}  in Lemma \ref{neolem} holds with $A=P_{2k+1}^2$ and $n=N$.
Let $x\in\llbracket N^{1/3},\varepsilon N \rrbracket$ and $m\in\llbracket 2x,2\varepsilon^{-1}x\rrbracket
\subset \llbracket 2N^{1/3},2N\rrbracket$.\\[0.5em]
-- If $2\varepsilon^{j+1} N \le m-x< 2\varepsilon^j N$ and $x\ge \varepsilon^{j+1}N$ then
either $m-2x\ge \varepsilon^{j+1}N$ or $\frac{m-x}2\ge m-2x$.
In both cases $\llbracket m-2x,m-x\rrbracket\cap\llbracket
\varepsilon^{j+1}N, 2\varepsilon^{j}N
\rrbracket$ has length $\ge x/2$.
 Hence
\begin{align*}
|P_{2k+1}^2\cap\llbracket m-2x,m-x\rrbracket| &  > 2
\left\lfloor
\frac{x}{\varepsilon^{j+2}N }\right\rfloor \varepsilon (\varepsilon^{j+2}N )^{2/3}\log \varepsilon^{-(j+1)}\\
&> \varepsilon^{2/3}x^{2/3}\log \left(\frac{N }x\right).
\end{align*}

\smallskip
\noindent
-- If $2\varepsilon^{j+1} N  \le m-x< 2\varepsilon^j N $ and $x <  \varepsilon^{j+1}N $
then 
$$
\llbracket m-2x,m-x\rrbracket\subset \llbracket\varepsilon^{j+1}N ,2\varepsilon^{j}N \rrbracket
$$
hence
\begin{align*}
|P_{2k+1}^2\cap\llbracket m-2x,m-x\rrbracket| &  > 2
\left\lfloor
\frac{2x}{\varepsilon^{j+2}N }\right\rfloor \varepsilon(\varepsilon^{j+2}N )^{2/3}\log \varepsilon^{-(j+1)}\\
&> \varepsilon x^{2/3}\log \left(\frac{N }x\right)\times\frac{2(j+1)}{j+2}\ge\varepsilon x^{2/3}\log \left(\frac{N }x\right)
\end{align*}
since $x\ge \frac{\varepsilon m}{2}>\varepsilon^{j+2}N$.

\medskip
Applying Lemma \ref{neolem} we obtain a partial additive complement $B_{2k+1}$ of $P_{2k+1}^2$ in $\llbracket 2N ^{1/3},2N \rrbracket$
such that 
\begin{equation}\label{B2k+1}
|B_{2k+1}|\ll \varepsilon^{-2/3}N^{1/3}.
\end{equation}
Moreover since for any $j$ there are $O(\varepsilon^{-2})$ intervals $I_{j,r}$ we deduce
\begin{align}\label{P2k+1}
|P_{2k+1}|\ll\sum_{j\ge0}\varepsilon^{-2}\sqrt{{\varepsilon} (\varepsilon^{j+2}  N )^{2/3}\log \varepsilon^{-(j+1)}}
&=\varepsilon^{-5/6}\sqrt{\log \varepsilon^{-1}} N ^{1/3}\sum_{j\ge0}\varepsilon^{j/3}\sqrt{j+1}\\
&\ll
\varepsilon^{-5/6}
\sqrt{\log \varepsilon^{-1}} N^{1/3}.\nonumber
\end{align}

\medskip
We define
$$
A_0=\{0,1\}\cup \bigcup_{k\ge1}\Big(\llbracket N_{2k-1},N_{2k}\rrbracket\cup P_{2k+1}\cup B_{2k+1}\Big).
$$
Notice that $S_k:=\llbracket N_{2k-1},N_{2k}\rrbracket\cup P_{2k+1}\cup B_{2k+1}\subset 
\llbracket N_{2k-1},N_{2k+1}\rrbracket$ hence the sets $S_k$'s do not overlap.
By \eqref{B2k+1}, \eqref{P2k+1} and since $N_{2k}=N_{2k+1}^{1/3}$,
we infer
\begin{equation}\label{eqp2k+1}
A_0(N_{2k+1})\le  N_{2k}+|P_{2k+1}|+|B_{2k+1}|\ll \varepsilon^{-5/6}\sqrt{\log \varepsilon^{-1}}N_{2k+1}^{1/3}.
\end{equation}
By Lemma \ref{neolem} again, 
$$
\left|\llbracket2N_{2k},t\rrbracket\smallsetminus (A_0^2+A_0)\right|\le \varepsilon t,\quad  \text{for any 
$2N_{2k}\le t\le 2N_{2k+1}$.}
$$
Furthermore we have $\llbracket 2N_{2k-1}, 2N_{2k}\rrbracket, 
\llbracket 2N_{2k+1}, 2N_{2k+2}\rrbracket\subset A_0+A_0\subset A_0^2+A_0$.  
Thus 
\begin{equation}\label{eqdens3}
\left|\llbracket 1 ,t\rrbracket\smallsetminus (A_0^2+A_0)\right|\le \varepsilon t+2N_{2k-1}=\varepsilon t+O(t^{1/3}),\quad  \text{for any 
$2N_{2k}<t\le 2N_{2k+2}$}.
\end{equation} 
We infer $\dinf{A_0^2+A_0}\ge 1-\varepsilon$.
\end{proof}

\begin{proof}[Proof of Theorem \ref{locthin}]
For deriving Theorem \ref{locthin} we  slightly modify the proof of Theorem \ref{thm53} by  letting $\varepsilon$ to be a function of $k$. We may assume that $\phi(t) < (\log\log t)^3$.
 For fixed $k$, we take $\varepsilon_k=\phi(N_{2k+1})^{-1/3}$. We check that 
 $N_{2k+1}^{1/9}\varepsilon_k^{-9/2}$ is big enough and that $\varepsilon_k > (\log\log N_{2k+1})^{-1}$, allowing us to 
  construct $P_{2k+1}$ as in the above proof using the Prime Number Theorem in slightly shorter intervals
  of the type $[X,X(1+(\log\log X))^{-1}]$.
  By \eqref{eqp2k+1}, \eqref{eqdens3} with $\varepsilon=\varepsilon_k$ and letting $k$ tend to infinity we deduce the required result.
\end{proof}

\section{\bf Probabilistic construction of a thin set $A$ such that $A^2+A^2=\mathbb{N}_0$}
\label{prob_S2+S2}

We define the probability measure $\mathbb{P} = \displaystyle\prod_{a \in \N}p_a$ on $Y= \{0,1\}^{\N}$ by choosing
$$
x_a:=p_a (\{1\})=\frac{c}{\sqrt{a}(\log (a+1))^{1/4}}.
$$
We shall choose $c< 1$ so that $x_a < 1$  and there exists such a probability measure.
The following result is easy to prove by partial summation.
\begin{lemma} \label{S2S2lambda} With the notations as above, we have
\begin{equation*} 
\lambda_n :=\sum_{a\leq n} x_a \sim 2c\sqrt{n}(\log n)^{-1/4}.
\end{equation*}
\end{lemma}

Hence using Lemma~\ref{S2S2lambda} and Corollary~\ref{card}, we obtain that
\begin{cor} \label{card22} For any $\varepsilon > 0$, we have
$$\mathbb{P}( \{A\subset \N: A(n) \sim  2c\sqrt{n}(\log n)^{-1/4} \text{ as $n$ tends to infinity }\}) =1.$$
\end{cor}
\medskip
We now define the random variable $R(n)$ counting the number of representations of $n$ under the form $n=ab+cd$  restricted to quadruples of distinct integers $a,b,c,d\in A$. In order to avoid repetitions, we also assume 
that $a<b$, $c<d$, $ab\le cd$:
$$
R(n)=\sideset{}{'}\sum_{a,b,c,d}\xi_a\xi_b\xi_c\xi_d
$$ 
where the dash indicates the above restrictions.
In the rest of this section, we shall prove the following result.

\begin{prop} \label{Rn22}For a suitable $c> 0,$ we have
$$\mathbb{P}(\{R(n) = 0\}) \leq \frac{1}{n^{1+\eta}}$$
for some $\eta > 0.$
\end{prop}
Using Corollary~\ref{card22}, Proposition~\ref{Rn22} and  Lemma~\ref{R(n)gen}, we obtain Theorem~\ref{thm13}.

\vspace{.5cm}
Let $\mathbb{P}'_n = \prod_{a=1}^n p'_a$ be a probability measure on $Y_n = \{0,1\}^{\N}$ with $p'_a(\{1\}) = \frac{c}{\sqrt{a}(\log (n+1))^{1/4}}.$ It is easy to see that $I_{(R(n) \neq 0)}$ is a monotone increasing function on $Y_n.$ Therefore to prove Proposition~\ref{Rn22}, using Lemma~\ref{transfert}, it is sufficient to prove that for a suitable $c > 0$, we have
\begin{equation*}
\mathbb{P}'_n(\{R(n) = 0\}) \leq \frac{1}{n^{1+\eta}}
\end{equation*}
for some $\eta > 0.$

\vspace{.5cm}
 Let $I = \{(a,b,c,d) \in \N^4: ab + cd =n; \; a, b, c, d \text{ distinct}, \ a<b,\ c<d,\ ab<cd\}$ be an index set and for any $(a,b,c, d) \in I$, with 
\begin{equation*}
Z_{(a,b,c,d)} = \xi_a\xi_b\xi_c\xi_d. ; \text{ we have } Z = \sum_{(a,b,c,d) \in I} Z_{(a,b,c,d)}.
\end{equation*}
Hence $R(n)$ is a sum of Boolean random variables. For $(a,b,c,d), (a',b', c', d') \in I$, the random variables $Z_{(a,b,c,d)}$ and $Z_{(a',b',c',d')}$ are independent if and only if $\{a,b,c,d\} \cap \{a', b', c', d'\}~=~\emptyset.$
Note that if $n=ab+cd=a'b'+c'd'$ and $(a,b, c, d) \neq (a',b',c',d'),$ then $\{a,b,c,d\}\ne \{a',b',c',d'\}$: indeed if
for instance $n=ab+cd=ac+bd$ then $a(b-c)=d(b-c)$ hence $a=d$ since $b\ne c$, a contradiction.

\medskip\noindent
Let $n\ge1$. For any quadruple of distinct positive integers $a,b,c,d$, we denote by $E_n(a,b,c,d)$ the event
$$
n=ab+cd 
 \quad \text{and}\quad \xi_a\xi_b\xi_c\xi_d=1.
$$
We observe that the events
$E_n(\sigma(a),\sigma(b),\sigma(c),\sigma(d))$, where $\sigma$ runs in the set of all permutations of $\{a,b,c,d\}$, are disjoint.
Moreover
$$
\mu_n:=\mathbb{E} (R(n))=\sideset{}{'}\sum_{a,b,c,d} \mathbb{P} (E_n(a,b,c,d)).
$$
where the dash in the summation means $a,b,c,d$ are distinct and $a<b$, $c<d$ and $ab<cd$.
If the events $E_n(a,b,c,d)$ where mutually independent we would have
$$
\mathbb{P} (R(n)=0)=\mathbb{P} \left(\bigcap_{a,b,c,d}\overline{E_n(a,b,c,d)}\right)
=\prod_{a,b,c,d}\left(1-\mathbb{P} (E_n(a,b,c,d)\right)\sim e^{-\mu_n}
$$
as $n$ tends to infinity. If $\mu_n\sim c'\log n$ as $n$ tends to infinity, with $c'>1$, then
we could deduce from Borel-Cantelli Lemma (cf. Lemma \ref{Borel-Cantelli}) that  almost surely $R(n)\ne0$ for any large enough $n$. \\
But the events $E_n(a,b,c,d)$ are not mutually independent, hence we need to measure their dependence.  We denote $(a,b,c,d)\sim(a',b',c',d')$ if $\{a,b,c,d\}\cap\{a',b',c',d'\}\ne\emptyset$ and
$(a,b,c,d)\ne(a',b',c',d')$. We are going to  concentrate on the estimation of 
$$
\Delta_n:=\sum_{(a,b,c,d)\sim(a',b',c',d')}\mathbb{P} \Big(E_n(a,b,c,d)\cap E_n(a',b',c',d')\Big).
$$
Our goal is to prove that $\mu_n\sim c'\log n$ and $\Delta_n=o(\log n)$. 
We will conclude 
by Janson's inequality (cf. Lemma \ref{Janson}).

\medskip
Let $\tau$ the divisor function. Our estimates will need the following classical facts: 
\begin{equation*}
\tau(m) \leq 2\sum_{\substack{l \leq \sqrt{m}\\ l\mid m}} 1,
\end{equation*}
Moreover for any $\varepsilon > 0$, we have 
$\tau(n)\ll_{\varepsilon} n^{\varepsilon}.$ Finally $\sum_{d\mid n}\frac1d\ll \log\log n$.

\medskip
We now come to our problem and start to estimate $\mu_n$ and $\Delta_n$.\\[0.5em]
Firstly by the next lemma (cf. Lemma \ref{estilem}) we have the lower bound
 \begin{equation}\label{mun}
 \mu_n=\sideset{}{'}\sum_{a,b,c,d}p_ap_bp_cp_d \sim  \frac{c^4}{8\log (n+1)} \sum_{0<k<n
 }\frac{\tau(k)\tau(n-k)}{\sqrt{n-k}\sqrt{k}} 
 > \left(\frac{3c^4}{4\pi} +o(1)\right) \log n.
 \end{equation}
 The factor $8$ in the denominator compensates  for the
 restrictions on $a,b,c,d$. 
 We used also the fact that  the contribution in the sum over $a,b,c,d$ in which two variables coincide
 is $O(\log\log n)$. Indeed:\\
 - when $a=b$ in $n=ab+cd$, the contribution is
 $$
 \ll \sum_{0<a<\sqrt{n}}\frac{\tau(n-a^2)}{a\sqrt{n-a^2}}\ll n^{\varepsilon}
 \sum_{0<a<\sqrt{n}}\frac{1}{a\sqrt{n-a^2}}.
 $$
 In the sum, for $0<a\le \frac{\sqrt{n}}2$, we get $O(\frac{\log n}{\sqrt{n}})$ ; for $\frac{\sqrt{n}}2<a<\sqrt{n}-1$, we get $O(\frac1{n^{1/4}})$; for $a=\lfloor \sqrt{n}\rfloor$ we get $O(\frac1{\sqrt{n}})$.\\
 - when $a=c$ the contribution is 
 $
 \sum_{a\mid n}\frac1a\sum_{0<b<\frac na}\frac1{\sqrt{b}\sqrt{\frac na-b}}\ll \log\log n$ by the easy estimate
 $$
 \sum_{0<k<n}\frac1{\sqrt{k(n-k)}}\sim \int_0^n\frac{dt}{\sqrt{t(n-t)}}=\pi +o(1).
 $$

 \begin{lem}\label{estilem}
 One has
 $$
 T_n:=\sum_{0<k<n}\frac{\tau(k)\tau(n-k)}{\sqrt{n-k}\sqrt{k}}\ge \frac{6+o(1)}{\pi}(\log n)^2, \quad n\to\infty.
 $$
 \end{lem}
 
 \begin{proof}[Proof of the lemma] We argue by partial summation, using the estimate due to Ingham (cf. \cite{In}).
 $$
\sum_{0<k<n}{\tau(k)\tau(n-k)}=\frac{6}{\pi^2}n(\log n)^2\sum_{q\mid n}\frac1q
\ge U(n):=\frac{6}{\pi^2}n(\log n)^2.
 $$
 We thus have
 $$
 \frac{T_n}2 \ge  \sum_{0<k<\frac n2}\frac{\tau(k)\tau(n-k)}{\sqrt{n-k}\sqrt{k}}
 \ge \frac{U(\frac n2)}{\frac n2}+\frac3{\pi^2} \int_1^{\frac n2}
 \frac{(n-2t)(\log t)^2}{\sqrt{t}(n-t)^{3/2}}dt.
 $$
 The above integral is equivalent to
 $$
(\log n)^2 \int_{1}^{\frac n2}\frac{(n-2t)}{\sqrt{t}(n-t)^{3/2}}dt
=(\log n)^2 \int_{1}^{\frac n2}\frac{dt}{\sqrt{t}\sqrt{n-t}} - 
(\log n)^2\int_{1}^{\frac n2}\frac{\sqrt{t}dt}{(n-t)^{3/2}}dt.
 $$
 By partial summation
 $$
 \int_{1}^{\frac n2}
 \frac{\sqrt{t}dt}{(n-t)^{3/2}}dt=2+o(1)-\int_{1}^{\frac n2}\frac{dt}{\sqrt{t}\sqrt{n-t}}
 $$
 hence the result since $\displaystyle\int_{0}^{ 1}\frac{du}{\sqrt{u}\sqrt{1-u}}=\pi$.
 \end{proof}

 \smallskip\noindent
 Secondly we observe that $(a,b,c,d)\sim(a',b',c',d')$ holds for only  $5$ different types of configurations:
\begin{enumerate}
\item[i)] $a=a'$ and  $a,b,c,d,b',c',d'$ are distinct

\item[ii)] $a=a'$, $b=b'$ and $a,b,c,d,c',d'$ are distinct

\item[iii)] $a=a'$, $c=c'$ and  $a,b,c,d,b',d'$ are distinct

\item[iv)] $a=a'$, $b=d'$  and  $a,b,c,d,c',d'$ are distinct

\item[v)] $a=a'$, $b=d'$, $c=c'$ and $a,b,c,d,d'$ are distinct 


\end{enumerate} 
 In the sequel we shall treat them separately and show that the corresponding contributions $E_i$, $i=1,\dots,5$, are negligible.
 
 \medskip
\noindent
Contribution (i). The representations of $n$ under the form $n=ab+cd=ab'+c'd'$ contribute for at most
$$
E_1\ll
\frac{1}{(\log n)^{7/4}}
\sum_{\substack{a,b,c,d,b',c',d'\\n=ab+cd=ab'+c'd'}}\frac1{\sqrt{abcdb'c'd'}}
= \frac{1}{(\log n)^{7/4}}\sum_{1 \leq a < n}\frac{1}{\sqrt{a}}\left(\sum_{b}\frac{\tau(n-ab)}{\sqrt{b(n-ab)}}\right)^2.
$$
 
 \begin{lemma}\label{lem65}
Let $a< b$ be real numbers and $a_1, \ldots, a_k \in [a, b]$ with $a_i - a_{i-1} \geq l > 0.$  If $f: (a-l, b] \to \R^+$ is a monotonically decreasing function, then
\begin{equation*}
f(a_1)+\ldots+f(a_k) \leq \frac{1}{l}\int_{a-l}^b f(u)du.
\end{equation*} 
If $f:[a, b+l)\to \R^+$ is a monotonically increasing function, then
 \begin{equation*}
f(a_1)+\ldots+f(a_k) \leq \frac{1}{l}\int_{a}^{b+l} f(u)du.
\end{equation*} 
\end{lemma}

\medskip
 We will readily derive $E_1\ll (\log n)^{1/4}$ from the following lemma.
\begin{lemma}
For any $a$, let $S_a = \displaystyle\sum_{b}\frac{\tau(n-ab)}{\sqrt{b(n-ab)}}$. Then
\begin{equation}\label{one-term}
 \sum_{a\leq n}\frac{S_a^2}{\sqrt{a}} \ll \log^2n.
\end{equation}
\end{lemma}

\begin{proof}[Proof of the lemma]
\noindent
For any $\varepsilon > 0$ and $a \geq n^{\varepsilon}$, we have $\tau(n-ab) \ll_{\varepsilon} a^{\varepsilon}$ with implied constant being independent of $b$ and depending only upon $\varepsilon$. Hence we have
\begin{align*}
\sum_{1 \leq b \leq \frac{n}{a} -1}\frac{\tau(n-ab)}{\sqrt{b(n-ab)}} &\ll_{\varepsilon} \frac{a^{\varepsilon}}{\sqrt{a}}\left(\sum_{1 \leq b \leq \frac{n}{2a}}\frac{1}{\sqrt{b}\sqrt{\frac{n}{a}-b}} + \sum_{ \frac{n}{2a}\leq b \leq \frac{n}{a}-1} \frac{1}{\sqrt{b}\sqrt{\frac{n}{a}-b}}  \right) \\
&\ll_{\varepsilon}  \frac{a^{\varepsilon}}{\sqrt{n}}\left(\sum_{1 \leq b \leq \frac{n}{2a}} \frac{1}{\sqrt{b}} + \sum_{ \frac{n}{2a}\leq b \leq \frac{n}{a}-1} \frac{1}{\sqrt{\frac{n}{a}-b}}\right) \\
& \ll_{\varepsilon} \frac{a^{\varepsilon}}{\sqrt{a}}.
\end{align*}
Hence we have
\begin{equation}\label{agepgen}
\sum_{n^{\varepsilon} \leq a \leq n} \frac{1}{\sqrt{a}} S_a^2 \ll_{\varepsilon} 1 + \sum_{n^{\varepsilon} \leq a \leq n} \frac{1}{\sqrt{a}}
\left(\sum_{\frac{n}{a}-1 < b \leq \frac{n}{a}} \frac{\tau(n-ab)}{\sqrt{b(n-ab)}}\right)^2.
\end{equation}
For any fixed $a$, there exists at most one integer $b_0 \in (\frac{n}{a} -1, \frac{n}{a}]$ and for such an integer $b_0$, let $k_a = n-ab_0.$ We have that $a$ divides $n-k_a$ and $b_0 \gg \frac{n}{a}.$ Hence we get
\begin{align}\label{agepspec}
\sum_{a\leq n}\frac{1}{\sqrt{a}}
\left(\sum_{\frac{n}{a}-1 < b \leq \frac{n}{a}} \frac{\tau(n-ab)}{\sqrt{b(n-ab)}}\right)^2 & \ll  \frac{1}{n}\sum_{a \leq n} \frac{\sqrt{a}d^2(k_a)}{k_a} \\
& \leq  \frac{1}{n}\sum_{ k\leq n}\frac{d^2(k)}{k}\sum_{a|n-k}\sqrt{a}\nonumber \\
&\ll_{\varepsilon}  \frac{n^{\varepsilon}}{\sqrt{n}}\log n.\nonumber
\end{align}
Using \eqref{agepgen}, \eqref{agepspec} and the inequality $(c+d)^2 \leq 2(c^2 + d^2)$, we obtain
\begin{equation}\label{agep}
\sum_{n^{\varepsilon} \leq a \leq n}\frac{S_a^2}{\sqrt{a}} \ll_{\varepsilon} 1.
\end{equation}
When $a \leq n^{\varepsilon}$ and $n$ is sufficiently large, we have
\begin{align*}
\sum_{\sqrt{n} < b < \frac{n}{a} -\sqrt{n}}\frac{\tau(n-ab)}{\sqrt{b(n-ab)}} & \leq  2 \sum_{l\leq \sqrt{n}}
\sum_{\substack{\sqrt{n} \leq b \leq \frac{n}{a} -\sqrt{n}\\
l\mid (n-ab)}} \frac{1}{\sqrt{b(n-ab)}}\nonumber \\
& \leq  2\sum_{\substack{l_1 \leq \sqrt{n}\\
l_1 | {\rm gcd}(a,n)}}
\sum_{\substack{l_2 \leq \frac{\sqrt{n}}{l_1}\\ {\rm gcd}(\frac{a}{l_1}, l_2)=1}}
\sum_{\substack{\sqrt{n} \leq b \leq \frac{n}{a} -\sqrt{n}\\
l_2\mid \frac{n-ab}{l_1}}}\frac{1}{\sqrt{b(n-ab)}}.\nonumber
\end{align*} 
Hence by Lemma \ref{lem65}
\begin{align}\label{alepgen}
\sum_{\sqrt{n} < b < \frac{n}{a} -\sqrt{n}}\frac{\tau(n-ab)}{\sqrt{b(n-ab)}}& \leq  2 \sum_{\substack{l_1 \leq \sqrt{n}\\
l_1 | {\rm gcd}(a,n)}}
\sum_{\substack{l_2 \leq \frac{\sqrt{n}}{l_1}\\ {\rm gcd}(\frac{a}{l_1}, l_2)=1}}\frac{1}{l_2}\int_0^{\frac{n}{a}}\frac{du}{\sqrt{u(n-au)}}\\
& \ll  \frac{d(a)}{\sqrt{a}}\log n.\nonumber
\end{align} 
When $a \leq n^{\varepsilon},$ we also have
\begin{align}\label{alnepsmallint} 
\sum_{b \leq \sqrt{n}}\frac{\tau(n-ab)}{\sqrt{b(n-ab)}} + \sum_{\frac{n}{a}-\sqrt{n}\leq b \leq \frac{n}{a}} \frac{\tau(n-ab)}{\sqrt{b(n-ab)}} &\ll_{\varepsilon} \frac{n^{\varepsilon}}{n^{1/4}} +\frac{\sqrt{a}n^{\varepsilon}}{\sqrt{n}}\sum_{m \leq a\sqrt{n}}\frac{1}{\sqrt{m}} \\
& \ll_{\varepsilon}  \frac{an^{\varepsilon}}{n^{1/4}}.\nonumber
\end{align}
Using
\eqref{alepgen} and \eqref{alnepsmallint}, we obtain that
\begin{equation}\label{alep}
\sum_{a\leq n^{\varepsilon}}\frac{1}{\sqrt{a}}S_a^2 \ll_{\varepsilon} \log^2 n + n^{(9\varepsilon -1)/2}.
\end{equation}
Using \eqref{agep} and \eqref{alep} with $\varepsilon = 1/9$, we obtain the result.
\end{proof}

\medskip
\noindent
Contribution (ii). The representations of $n$ under the form $n=ab+cd=ab+c'd'$ contribute for at most
$$
E_2\ll
\frac{1}{(\log n)^{3/2}}
\sum_{\substack{a,b,c,d,c',d'\\n=ab+cd=ab+c'd'}}\frac1{\sqrt{abcdc'd'}}
$$
Letting $h=ab$, the inner sum becomes
$$
\sum_{0<h<n}\frac{\tau(h)\tau(n-h)^2}{\sqrt{h}(n-h)}\ll n^{\varepsilon-1/2}.
$$

\medskip
\noindent
Contribution (iii). We have $n=ab+cd=ab'+cd'$. Let $q=\gcd(a,c)$. Then $q\mid n$ and
$$
\frac nq=\alpha b+\gamma d=\alpha b'+\gamma d',\quad \alpha =\frac aq,\quad \gamma =\frac cq.
$$
Let $\alpha ,\gamma $ fixed.
Since $\gcd(\alpha ,\gamma )=1$ we have $\alpha \mid (d-d')$ and $\gamma \mid (b-b')$. Let $(b_{\alpha ,\gamma },d_{\alpha ,\gamma })$ a fixed solution of the equation
$\frac nq=\alpha x+\gamma y$. Then there exists $\lambda\in\mathbb{Z}$ such that $(b,d)=(b_{\alpha ,\gamma }-\lambda \gamma ,d_{\alpha ,\gamma }+\lambda \alpha )$. Similarly $(b',d')=(b_{\alpha ,\gamma }-\mu \gamma ,d_{\alpha ,\gamma }+\mu \alpha )$ for some integer $\mu$. When $b,d$ run in $(0,n)\cap\mathbb{N}$ according to the given restrictions, $\lambda$ runs in some interval $I_{\alpha ,\gamma }$. Further there exists at most a $\lambda_0\in I_{\alpha ,\gamma }$ such that $b_{\alpha ,\gamma }-\lambda_0 \gamma  < \frac{\gamma }2$ and at most a $\lambda_1 \in I_{\alpha ,\gamma }$ such that $d_{\alpha ,\gamma }-\lambda_1 \alpha  < \frac{\alpha}2$. The contribution corresponding to case (3) is
$$
\ll E_3:= \frac{1}{(\log n)^{3/2}}\sum_{q\mid n}\frac1q\sum_{\alpha ,\gamma }\frac{1}{\sqrt{\alpha }\sqrt{\gamma }}
\sum_{\substack{b,d,b',d' \\ n=q(\alpha b+\gamma d)=q(\alpha b'+\gamma d')}}\frac{1}{\sqrt{b}\sqrt{d}\sqrt{b'}\sqrt{d'}}.
$$ 
The inner sum can be rewritten and bounded by 
$$
\frac1{\alpha \gamma }
\sum_{\substack{\lambda, \mu\in I_{\alpha ,\gamma }\\\lambda\ne \mu}}
\left(\frac{b_{\alpha ,\gamma }}{\gamma }-\lambda \right)^{-1/2}
\left(\frac{d_{\alpha ,\gamma }}{\alpha }+\lambda \right)^{-1/2}
\left(\frac{b_{\alpha ,\gamma }}{\gamma }-\mu \right)^{-1/2}
\left(\frac{d_{\alpha ,\gamma }}{\alpha }+\mu \right)^{-1/2}.
$$
For brevity  let $F(\lambda,\mu)$ be denote the summand in the above double sum.  
Observe also that  $\lambda_0  \ge0\ge \lambda_1$ hence $\lambda_0=\lambda_1$ only if their common value is $0$ in which case $I_{\alpha ,\gamma }=\{0\}$. Hence in that case  the summation over $\lambda,\mu$ is empty and the corresponding contribution is zero.

We now assume $\lambda_0  > \lambda_1$.
By developing the sum over $\lambda,\mu$ we obtain
$$
\sum_{\substack{\lambda,\mu \in I_{\alpha ,\gamma }\\ \lambda,\mu \ne \lambda_0,\lambda_1\\ \lambda\ne \mu}}
F(\lambda,\mu)
+
2\sum_{i=0}^1\sum_{\substack{\lambda\in I_{\alpha ,\gamma }\\
\lambda\ne \lambda_0,\lambda_1}}F(\lambda,\lambda_i)
+2F(\lambda_0,\lambda_1).
$$
The first sum involves 
\begin{equation}\label{eqlambda}
 \sum_{\substack{\lambda\in I_{\alpha ,\gamma }\\\lambda\ne \lambda_0,\lambda_1}}
\left(\frac{b_{\alpha ,\gamma }}{\gamma }-\lambda \right)^{-1/2}
\left(\frac{d_{\alpha ,\gamma }}{\alpha }+\lambda \right)^{-1/2}=O(1)
\end{equation}
by the next lemma (with absolute constant). Hence
$$
\sum_{\substack{\lambda,\mu \in I_{\alpha ,\gamma }\\ \lambda\ne \lambda_0,\lambda_1\\\lambda\ne \mu}}
F(\lambda,\mu)\ll1.
$$

\begin{lemma}\label{lem41}
Let $u, v$ two positive real numbers. Then
$$
\sum_{\frac12-v\le j\le u-\frac 12}\frac{1}{\sqrt{(u-j)(v+j})}\le 12.
$$ 
\end{lemma}
\begin{proof}[Proof of the lemma]
The sum splits into 3 terms according to the range covered  by $j$: for $\frac12-v\le j\le 1$, $j=0$ and $1\le j\le u-1$. The variable $j$ takes the value $0$ only if $u,v\ge\frac12$ hence $\frac1{\sqrt{uv}}\le 4$. The first and the third cases are similar. Letting $f_{u,v}(j)$ the summand in the considered sum, one has
\begin{align*}
\sum_{0\le j\le u-\frac12}f_{u,v}(j) & =\sum_{1\le j \le   \frac u2}f_{u,v}(j)+\sum_{\frac u2 <  j \le u-\frac12}f_{u,v}(j)\\
& \le \frac1{\sqrt{\frac u2}}\sum_{1\le j \le   \frac u2}\frac1{\sqrt{v+j}}+
\frac1{\sqrt{v+\frac u2}}\sum_{1\le j <  \frac u2}\frac1{\sqrt{j}}\\
&\le  \frac2{\sqrt{\frac u2}}\sum_{1\le j \le   \frac u2}\frac1{\sqrt{j}}\le \frac2{\sqrt{\frac u2}}\int_{0}^{\frac u2}\frac{dt}{\sqrt{t}}=4,
\end{align*}
and the bound follows.
\end{proof}

Since $\frac{b_{\alpha ,\gamma }}{\gamma }-\lambda_0\ge\frac1\gamma $ and $\frac{d_{\alpha ,\gamma }}{\alpha }+\lambda_0\ge\frac1\alpha $ we get
$$
\sum_{i=0}^1 \sum_{\substack{\lambda\in I_{\alpha ,\gamma }\\\lambda\ne \lambda_0,\lambda_1}}F(\lambda,\lambda_i)
=O(\sqrt{\gamma }+\sqrt{\alpha })
$$
 where we use again \eqref{eqlambda}.
 Finally $F(\lambda_0,\lambda_1)=O(\sqrt{\alpha \gamma })$ since $\lambda_0\ne\lambda_1$. 
 
 \medskip
 This readily gives
 \begin{align*}
 E_3 &\ll \frac{1}{(\log n)^{3/2}}\sum_{q\mid n}\frac1q\sum_{0<\alpha ,\gamma <\frac nq}\frac{1}{\alpha ^{3/2}\gamma ^{3/2}}(1+\sqrt{\alpha }+\sqrt{\gamma }+\sqrt{\alpha \gamma })\\
 & \ll  \sqrt{\log n}\log\log n.
 \end{align*}
 
 \medskip
\noindent
Contribution (iv). Here $n=ab+cd=ab'+c'b$, hence these representations contribute for
$$E_4\ll 
\frac{1}{(\log n)^{3/2}}
\sum_{\substack{a,b,c,d,b',c'\\n=ab+cd=ab'+c'b
}}\frac1{\sqrt{abcdb'c'}}.
$$
For any $a,b$, one has $q=\gcd(a,b)\mid n$. Further $q^2\mid ab<n$ thus $q<\sqrt{n}$. Hence 
$$
E_4 \ll \frac{1}{(\log n)^{3/2}}\sum_{\substack{q\mid n\\q<\sqrt{n}}}\sum_{\substack{a,b\\ \gcd(a,b)=q}}
\frac{\tau(n-ab)}{\sqrt{n-ab}}\sum_{\substack{a\mid k\\b\mid(n-k)}}\frac1{\sqrt{k}\sqrt{n-k}}.
$$
We fix $a,b$ and denote by $K_{a,b}$ the smallest positive integer such that 
$a\mid K_{a,b}$ and $b\mid (n-Kq)$. Let also $\lambda_{a,b}=(qn-q^2K_{a,b})/ab$. 
Then the inner sum in the above inequality is
$$
\frac1q\sum_{0\le \lambda\le \lambda_{a,b}}\left(
K_{a,b}+\lambda\frac{ab}{q^2}\right)^{-1/2}
\left(
\frac nq -K_{a,b}-\lambda\frac{ab}{q^2}\right)^{-1/2}.
$$ 
This sum restricted to $0<\lambda<\lambda_1$ is bounded from Lemma \ref{lem41} by $O\big(\frac{q^2}{ab}\big)$. Letting $f(\lambda)$ the summand in the above sum we  obtain the bound
$$
\sum_{\substack{a\mid k\\b\mid(n-k)}}\frac1{\sqrt{k}\sqrt{n-k}} \ll \frac{q}{ab} +\frac{f(0)+f(\lambda_{a,b})}q.
$$
This yields 2 types of contribution for $E_4$, those given by $f(0)/q$ and $f(\lambda_{a,b})/q$ being treated similarly.
The first one is
$$
E'_4=\frac{1}{(\log n)^{3/2}}\sum_{\substack{q\mid n\\q<\sqrt{n}}}\sum_{\substack{a,b\\ \gcd(a,b)=q}}
\frac{\tau(n-ab)}{\sqrt{n-ab}}\times \frac{q}{ab}
\ll n^{\varepsilon}\sum_{\substack{q\mid n\\q<\sqrt{n}}}\frac1{q^2}
\sum_{h<\frac{n}{q^2}}\frac{1}{h}\left(\frac{n}{q^2}-h\right)^{-1/2}.
$$
Since $q\mid n$ the fractional part of $\frac{n}{q^2}\ne0$ is $\ge \frac 1q$.
Separating the case $h=\big\lfloor \frac{n}{q^2}\big\rfloor$ from the rest of the sum over $h$ we find that it is
$\ll \frac{q^3}{n}+\frac{q\log n}{\sqrt{n}}$. It follows that
$$
E'_4\ll_{\varepsilon} \frac{n^{\varepsilon}}{n}\sum_{\substack{q\mid n\\q<\sqrt{n}}}q
+\frac{n^{\varepsilon}}{\sqrt{n}}\sum_{\substack{q\mid n\\q<\sqrt{n}}}\frac1q
\ll_{\varepsilon} \frac{n^{\varepsilon}}{\sqrt{n}}.
$$
For the remaining contribution and by symmetry  we only have to consider that coming from the term $f(0)$. 
By definition of $K_{a,b}$, the product $\frac{K_{a,b}}{aq^{-1}}\left(\frac{nq^{-1}-K_{a,b}}{bq^{-1}}\right)$ is a positive integer, hence
we have
\begin{align*}
{(\log n)^{3/2}}E''_4 & \ll \sum_{\substack{q\mid n\\q<\sqrt{n}}}\frac1q\sum_{\substack{a,b\\ \gcd(a,b)=q}}
\frac{\tau(n-ab)}{\sqrt{n-ab}}
\times \frac{q}{\sqrt{ab}}\left(
\frac{K_{a,b}}{aq^{-1}}\left(\frac{nq^{-1}-K_{a,b}}{bq^{-1}}
\right)
\right)^{-1/2}\\ 
&\le
\sum_{\substack{q\mid n\\q<\sqrt{n}}}\frac1q\sum_{\substack{a,b\\ \gcd(a,b)=q}}
\frac{\tau(n-ab)}{\sqrt{n-ab}}
\times \frac{q}{\sqrt{ab}}=\sum_{a,b}\frac{\tau(n-ab)}{\sqrt{n-ab}\sqrt{ab}}= \mu_n \log n.
\end{align*}
Hence $E_4\ll \dfrac{\mu_n}{\sqrt{\log n}}$.


 \medskip
\noindent
Contribution (v). We have $n=ab+cd=ab'+cb$. Hence $b'$ is uniquely determined by the other variables. This yields the bound for the contribution
$$
E_5\ll
\frac{1}{(\log n)^{5/4}}
\sum_{\substack{a,b,c,d,b'\\n=ab+cd=ab'+cb}}\frac1{\sqrt{abcdb'}}\le \frac{1}{(\log n)^{3/2}}
\sum_{\substack{a,b,c,d\\n=ab+cd}}\frac1{\sqrt{abcd}}\ll \frac{\mu_n}{(\log n)^{1/4}}.
$$
We conclude that 
\begin{equation}\label{ubdelta}
\Delta_n\ll\sum_{i=1}^5 E_i\ll \frac{\mu_n}{(\log n)^{1/4}}.
\end{equation} 

\medskip
It thus follows that if $3c^4\pi/4>1$, almost surely the random set $A $ has  counting function $A(n)\sim 2c n^{1/2}(\log n)^{-1/4}$ and satisfies $\mathbb{N}\setminus(A^2+A^2)$ is finite. By completing   if necessary $A$ by a finite number of  nonnegative integers, we get the announced result in Theorem \ref{thm13}: we state it under the sharpest following form (the constant is the best possible provided by this probabilistic approach):
   
   \begin{thm} Let $c>(\frac{4\pi}3)^{1/4}$.
   There exists a set of integers
   $A$ such that $A^2+A^2=\mathbb{N}_0$ and  $A(X)\sim \frac{2c\sqrt{X}}{(\log X)^{1/4}}$ as $X\to\infty$.
   \end{thm}
   
\begin{rmk}
Let $l\in\N$.
Theorem \ref{limsuplb} can be extended and Theorem \ref{thm13} can be straight generalized to the sum-product set 
 $$
 \Sigma_{l,2}(A):=\underbrace{A+\cdots+A}_{l \text{ times}}+A^2+A^2.
 $$
 Namely there exists a set $A\subset\mathbb{N}_0$ such that $A(X)\ll \frac{X^{1/(l+2)}}{(\log X)^{1/(l+4)}}$ and $ \Sigma_{l,2}(A)=\mathbb{N}_0$. We do not provide the complete proof, we only point out the main points. Since we are no longer concerned with the constant, we may assume
 that all the $l+2$ summands in $n=x_1+\cdots+x_l+ab+cd$ satisfy $x_i,ab,cd\asymp n$. The elementary
 probability for $x\in\mathbb{N}$ is given by $p'_x=\dfrac{c}{x^{(l+1)/(l+2)}(\log n)^{1/(l+4)}}$.
 Then using plain notation the expectation
 $\mu_n=\E(R(n))$  is
$$
 \mu_n\gg \frac1{n^{(l+2)(l+1)/(l+2)}\log n}\sum_{\substack{x_1,\dots,x_l,h,k\asymp n\\n=x_1+\cdots+x_l+h+k}}\tau(h)\tau(k)\gg \log n.
$$
The estimation of $\Delta_n$ concerns variable coincidences inside both representations
$$n=x_1+\cdots+x_l+ab+cd=x'_1+\cdots+x'_l+a'b'+c'd'.$$
Each collision $x_i$ with some variable in the second representation induces a lesser degree of freedom in the summation with the counterpart that a factor $n^{-(l-1)/l}$  
is cleared. There could be an additional $n^{\varepsilon}$ coming from the divisor function when for instance $x_i=a'$. It gives  a contribution to $\Delta_n$ being $\ll n^{-1/l+\varepsilon}\mu_n$. 
\\
In case of a unique collision among $a,b,c,d$ and $a',b',c',d'$, we consider $n=x_1+\cdots+x_l+ab+cd=x'_1+\cdots+x'_l+ab'+c'd'$. Letting $h=ab$, $k=cd$, $h'=ab'$ and $k'=c'd'$, the related contribution reduces to
$$
\ll \frac1{n^4(\log n)^{2-1/(l+4)}}\sum_{h,h',k,k'}\tau(k)\tau(k')\sum_{a\mid \gcd(h,k)}a^{(l+1)/(l+2)}\\
\ll \frac{(\log n)^{1/(l+4)}}{n^2}\sum_{h,h'}\sum_{a\mid \gcd(h,k)}a^{(l+1)/(l+2)}
$$
Inverting the summations gives $O(n^2)$ for the triple sum, hence a total contribution $\ll (\log n)^{1/(l+4)}=o(\mu_n)$. The remaining cases with $2$, $3$ or $4$ collisions are easy to consider and yields smaller contributions. We infer $\Delta_n=o(\mu_n)$.
\end{rmk}

\begin{rmk} Let $\delta > 0.$
The arguments used to prove Theorem~\ref{thm13} can be used to prove the existence of $A \subset \N$ such that for any sufficiently large $n\in \N$ and $x\in \R,$ we have

$$ n= ab + cd, \; a,b, c, d \in A, \text{ with }\; d \leq n^{\delta}\; \text{ and } A(x)\ll_{\delta} \frac{x^{1/2}}{\log^{1/4}x}.$$
We do not provide the complete proof and only point out the main points. Since we are no longer concerned with the constant, we may assume
 that in $n= ab+cd$ satisfy $ab,cd\asymp n$ and $d \leq n^{\delta}.$ The elementary
 probability for $x\in\mathbb{N}$ is given by~$p'_x~=~\dfrac{c}{x^{1/2}(\log n)^{1/4}}$.
Let 

$$ R(n) = \sideset{}{'}\sum_{ab+cd =n} \xi_a\xi_b\xi_c\xi_d,$$
where the dash in the above summation indicates the restriction $a,b,c, d$ being distinct and 
$ab,cd\asymp n$ with $d \leq n^{\delta}.$ We have the following lower bound   
$$\mu_n:=\E(R(n)) \gg \frac{c^4}{n\log n} \sum_{h+k = n, h,k \asymp n}\tau(h) \tau_{\delta}(k), $$
where $\tau_{\delta}(k) = \sum_{d| k, d \leq n^{\delta}} 1.$ Assuming that $\delta \leq 1/2,$ using the lower bound $\tau(h) \geq \sum_{ a|h, a \leq n^{1/4}}1$ we obtain that $\mu_n \geq c(\delta) c^4 \log n,$ where $c(\delta) > 0$ is a constant depending only upon $\delta.$  We choose $c$ such that $c(\delta) c^4 > 1.$ For the purpose of obtaining an upper bound for $\Delta_n$, we may ignore the condition that $d \leq n^{\delta}$ and use directly the bound provided by~\eqref{ubdelta} to obtain that $\Delta_n \ll \log^{3/4}n =o(\mu_n).$
\end{rmk}

\end{document}